\documentclass[12pt]{amsart}
\usepackage[utf8]{inputenc}
\usepackage{enumerate, amsmath, amsthm, amsfonts, amssymb,  mathrsfs, graphicx, paralist, fancyvrb}
\usepackage[usenames, dvipsnames]{xcolor}
\usepackage[margin=1in]{geometry} 
\usepackage[bookmarks, colorlinks=true, linkcolor=blue, citecolor=blue, urlcolor=blue]{hyperref}
\usepackage{tikz-cd}
\usepackage{mathtools}
\usepackage[all,cmtip]{xy}
\input xy
\xyoption{all}

\numberwithin{equation}{subsection}

\newtheorem{thm}{Theorem}
\newtheorem{proposition}[equation]{Proposition}
\newtheorem{lemma}[equation]{Lemma}

\theoremstyle{definition}
\newtheorem{rmk}[equation]{Remark}

\newtheorem{eg}[equation]{Example}

\newtheorem{defn}[equation]{Definition}

\makeatletter
\@addtoreset{equation}{section}
\renewcommand{\thesubsection}{%
  \ifnum\c@subsection<1 \@arabic\c@section
  \else \thesection.\@arabic\c@subsection
  \fi
}
\makeatother

\newcommand{\fg}{\mathfrak{g}}

\newcommand{\ft}{\mathfrak{t}}

\title{An Explicit Determination of the Springer Morphism}
\author{Sean Rogers}
\date{October 2016}

\begin{document}

\maketitle
\begin{abstract}
    Let $G$ be a simply connected semisimple algebraic groups over $\mathbb{C}$ and let $\rho :G\rightarrow GL(V_\lambda)$ be an irreducible representation of $G$ of highest weight $\lambda$. Suppose that $\rho$ has finite kernel.Springer defined an adjoint-invariant regular map with Zariski dense image from the group to the Lie algebra, $\theta_\lambda:G\rightarrow\fg$, which depends on $\lambda$ [BP,\S 9]. By a lemma in [Kum] $\theta_\lambda$ takes the maximal torus to its Lie algebra $\ft$. Thus, for a given simple group $G$ and an irreducible representation $V_\lambda$, one may write $\theta_\lambda (t)=\sum\limits_{i=1}^n c_i(t)\check{\alpha_i}$, where we take the simple coroots $\{\check{\alpha_i}\}$ as a basis for $\ft$. We give a complete determination for these coefficients $c_i(t)$ for any simple group $G$ as a sum over the weights of the torus action on $V_\lambda$.
\end{abstract}
\section{Introduction}
Let $G$ be a connected reductive algebraic group over $\mathbb{C}$ with Borel subgroup $B$ and maximal torus $T\subset B$ of rank $n$ with character group $X^*(T)$. Let $P$ be a standard paraoblic subgroup with Levi subgroup $L$ containing $T$. Let $W$ (resp. $W_L$) be the Weyl group of $G$ (resp. $L$). Let $V_\lambda$ be an irreducible almost faithful representation of $G$ with highest weight $\lambda$, i.e. $\lambda$ is a dominant integral weight and the corresponding map $\rho_\lambda :G\rightarrow Aut(V_\lambda)$ has finite kernel. Then, Springer defined an adjoint-invariant regular map with Zariski dense image from the group to its Lie algebra, $\theta_\lambda:G\rightarrow\fg$, which depends on $\lambda$ (Sect. 2.1). 

In recent work by Kumar [Kum], the Springer morphism is used in a crucial way to extend the classical result relating the polynomial representation ring of the general linear group $GL_r$ and the singular cohomology ring $H^*(Gr(r,n))$ of the Grassmanian of r-planes in $\mathbb{C}^n$ to the Levi subgroups of any reductive group $G$ and the cohomology of the corresponding flag varieties $G/P$. Computing $\theta_\lambda|_T$ is integral to this process. By a  lemma in [Kum], $\theta_\lambda$ takes the maximal torus $T$ to its Lie algebra $\ft$, thus inducing a $\mathbb{C}$-algebra homomorphism $(\theta_\lambda|_T)^*:\mathbb{C}[\ft]\rightarrow \mathbb{C}[T]$ between the corresponding affine coordinate rings. The Springer morphism is adjoint invariant and thus $(\theta_\lambda|_T)^*$ takes $\mathbb{C}[\ft]^{W_L}$ to $\mathbb{C}[T]^{W_L}$. One can then define the $\lambda-polynomial\ subring\ Rep^\mathbb{C}_{\lambda-poly}(L)$ to be the image of $\mathbb{C}[\ft]^{W_L}$ under $(\theta_\lambda|_T)^*$ (as $Rep^\mathbb{C}(L)\simeq \mathbb{C}[T]^{W_L}$). This leads to a surjective $\mathbb{C}$-algebra homomorphism $\xi^P_\lambda:Rep^\mathbb{C}_{\lambda-poly}(L)\rightarrow H^*(G/P,\mathbb{C})$, as in [Kum]. The aim of this work is to compute $\theta_\lambda|_T$ in a uniform way for all simple algebraic groups $G$ and any dominant integral weight $\lambda$.

As $\theta_\lambda|_T$ maps $T$ into $\ft$, we have that for a given simple group $G$ and an irreducible representation $V_\lambda$, one may write

$$\theta_\lambda (t)=\sum\limits_{i=1}^n c_i(\lambda)\check{\alpha_i}$$

, where we take the simple coroots $\{\check{\alpha_i}\}$ as a basis for $\ft$. We give a complete determination for these coefficients $c_i(t)$ for any simple, simply-connected algebraic group $G$ as a sum over the weights of the torus action on $V_\lambda$. For a given representation $V_\lambda$, let $\Lambda_\lambda$ be the set of weights appearing in the weight space decomposition of $V_\lambda=\bigoplus V_\mu$, listed with multiplicity. Let $\omega_1,...,\omega_n$ be the fundamental weights in $\ft^*$, and consider the weights $\mu\in\Lambda_\lambda$ written in the fundamental weight basis, i.e. $\mu=(\mu_1,...,\mu_n)=\mu_1\omega_1+...+\mu_n\omega_n$. Let $e^\mu(t)\in X^*(T)$ be the corresponding character of $T$. Then we find (Sect. 3) that,

\begin{thm} The coefficients $c_i(t)$ are determined by the following set of equations.
$$\begin{pmatrix}
\sum\limits_{\mu\in\Lambda_\lambda}\mu_1\cdot e^\mu (t)\\
\vdots\\
\sum\limits_{\mu\in\Lambda_\lambda}\mu_n\cdot e^\mu(t) \end{pmatrix}=
S(G,\lambda)
 \begin{pmatrix}
 c_1(t)\\c_2(t)\\ \vdots \\ c_n(t)
 \end{pmatrix},$$
 where $S(G,\lambda)=\{\sum\limits_{\mu\in\Lambda_\lambda}\mu_i\mu_j\}_{ij}$.
 \end{thm}
 Our main result (Sect. 4) determines that
 \begin{thm}The above matrix
 
 $$S(G,\lambda):=\{\sum\limits_{\mu\in\Lambda_\lambda}\mu_i\mu_j\}_{ij}=(\frac{1}{2}\sum\limits_{\mu\in\Lambda_\lambda} \mu_i^2) S\ ,$$
 
  where $S$ is a symmetrization of the Cartan matrix $A$ for $G$, and $\mu_i$ is the coordinate of the fundamental weight corresponding to a long root (or in the simply-laced case any root).
 \end{thm}
 In particular, for the simply-laced groups $S(G,\lambda)=(\frac{1}{2}\sum\limits_{\mu\in\Lambda_\lambda}\mu_1^2) A$. The determination of $S(G,\lambda)$ relies on the fact that $\Lambda_\lambda$ is invariant under the action of the Weyl group $W$, and moreover that if $\sigma\in W$ then $dim(V_\mu)=dim(V_{\sigma.\mu}).$

\section{Preliminaries}
Let $G$ be a simply-connected semi-simple algebraic group over $\mathbb{C}$, with Lie algebra $\fg=\ft\oplus\bigoplus\limits_\alpha \fg_\alpha$ of rank $n$, and  fixed base of simple roots $\Delta=\{\alpha_j\}$. Take the set of simple co-roots $\check{\Delta}=\{\check{\alpha}_j\}$ as a basis for the Cartan subalgebra $\ft\subset \fg$. Then $\ft_\mathbb{Z} =\bigoplus\limits_{j=1}^{n} \mathbb{Z}\check{\alpha}_j$ is the co-root lattice. Further, the weight lattice is $\ft_{\mathbb{Z}}^*=\bigoplus\limits_{i=1}^{n} \mathbb{Z}\omega_i$, where $\omega_i \in \ft^*$ is the $i^{th}$ fundamental weight of $\fg$ defined by $\omega_i (\check{\alpha}_j)=\delta_{ij}$. Then the maximal torus $T\subset G$ (with Lie algebra $\ft$) can be identified with $T=Hom_{\mathbb{Z}}(\ft_{\mathbb{Z}}^*,\mathbb{C}^*)$ as in [Kum2]. Finally, let $W$ be the Weyl group of $G$, generated by the simple reflections $s_i$. So for $\mu\in\ft^*$, $s_i(\mu)=\mu -\mu(\check{\alpha_i})\alpha_i$.

Let $V_\lambda$ be the irreducible representation of $G$ with highest weight $\lambda$. Then $V_\lambda$ has weight space decomposition
$$V_\lambda=\bigoplus V_{\mu}$$
where $V_{\mu_1,\mu_2,...,\mu_n}=\{v\in V_\lambda |\ t.v=((\mu_1\omega_1 +...+\mu_n\omega_n)(t))v\ \forall v\in V_\lambda\}$ is the weight space with weight $\mu=\mu_1\omega_1 +...+\mu_n\omega_n$.

So for $t\in T$ and $v\in V_{\mu_1,\mu_2,...,\mu_n}$ we have that the action of $t$ on $v$ is given by  
$$t.v=t(\mu_1,...,\mu_n)v=e^\mu(t)v$$ 
where $(\mu_1,...\mu_n)=\mu_1\omega_1 +...+\mu_n\omega_n$. Additionally $\check{\alpha}_j \in \ft$ acts on $v$ by 
$$\check{\alpha}_j.v=(\mu_1\omega_1 +...+\mu_n\omega_n)(\check{\alpha_j})v=\mu_jv.$$

\subsection{Springer Morphism}
For a given almost faithful irreducible representation $V_\lambda$ of $G$ we define the Springer morphism as in  [BP]
$$\theta_\lambda :G\rightarrow \fg$$ given by 

$$\xymatrix{
G \ar[r] \ar[dr]^{\theta_{\lambda}} &Aut(V(\lambda))\subset 
 End(V(\lambda))=\fg\oplus\fg^{\bot}\ar[d]^{\pi} \\
&\fg}$$

where $\fg$ sits canonically inside $End(V_\lambda)$ via the derivative $d\rho_\lambda$, the orthogonal complement $\fg^{\bot}$ is taken via the adjoint invariant form $<A,B>=tr(AB)$ on $End(V_\lambda)$, and $\pi$ is the projection onto the $\fg$ component. Note, that since $\pi\circ d\rho_\lambda$ is the the identity map, $\theta_\lambda$ is a local diffeomorphism at 1. Since the decomposition $End(V_\lambda)=\fg\oplus\fg^{\bot}$ is G-stable, $\theta_\lambda$ is invariant under conjugation in $G$. Importantly, $\theta_\lambda$ restricts to $\theta_{\lambda|T}:T\mapsto \ft$. [Kum]

\section{General Case}
Let $V_\lambda$ be a $d$ dimensional almost faithful irreducible representation of $G$ of highest weight $\lambda$. Let $\Lambda_\lambda=\{(\mu_1^i,...,\mu_n^i)\}_{i=1}^d$ be an enumeration of the set of weights considered with their multiplicity that appear in the weight space decomposition of $V_\lambda$ (so $\mu_j^i$ is the coordinate of the $j^{th}$ fundamental weight for the $i^{th}$ weight in the decomposition)  Then  we can take a basis of weight vectors $\{v_{\mu_1^i,...,\mu_n^i}\}_{i=1}^{d}$ on which the torus $T$ and each simple co-root acts diagonally. Thus, 

$$\rho_\lambda(t)=diag\{e^{\mu^1}(t),...,e^{\mu^d}(t)\}\in Aut(V_\lambda)$$ 

and for a simple co-root $\check{\alpha}_j$ we have that

$$d\rho_\lambda(\check{\alpha}_j)=diag\{\mu_j^1,...,\mu_j^d\}\in End(V_\lambda).$$

To take the projection we calculate $d\rho_\lambda(\fg)^{\bot}\in End(V_\lambda)$ with respect to the symmetric bilinear form $tr(AB)$. So letting $X=(x_{ij})$ be a $d\times d$ matrix in $End(V_\lambda)$ we have that for any co-root $\check{\alpha}_j\in\ft$ we require that 

   $$ tr(d\rho_\lambda(\check{\alpha}_j)\cdot X)=0 \implies \sum\limits_{i=1}^{d} \mu_j^i x_{ii}=0$$
   in order for $X\in d\rho_\lambda(\fg)^{\bot}.$

So $\sum\limits_{\Lambda_\lambda}\mu_1^i x_{ii}=\sum\limits_{\Lambda_\lambda}\mu_2^i x_{ii}=...=\sum\limits_{\Lambda_\lambda}\mu_n^i x_{ii}=0$. Now to project $\rho_\lambda(t)$ onto $d\rho_\lambda(\ft)$ we write $\rho_\lambda$ as a sum

$$\rho_\lambda(t)=\sum\limits_{j=1}^{n} c_j(t) d\rho_\lambda(\check{\alpha}_j) + X(t).$$

where $c_j:T\mapsto\mathbb{C}$ is a function that depends on $\lambda$, and $X(t)\in d\rho_\lambda(\fg)^{\bot}$. It follows then that 
$$\theta_\lambda(t)=\sum c_j(t)\check{\alpha_j}$$

So we aim to solve for the coefficients $c_j(t)$. Note that for the root space $\fg_\alpha$, we have that $\fg_\alpha .V_\mu \subset V_{\mu +\alpha}$. Thus, $d\rho_\lambda(e_\alpha)$ for $e_\alpha\in\fg_\alpha$ will only have off diagonal entries, and as such the condition $tr(d\rho_\lambda(e_\alpha)\cdot X)=0$ will only add constraints to the off diagonal entries of $X\in d\rho_\lambda(\fg)^{\bot}$.  As the action of $t$ and $\check{\alpha_j}$ are both diagonal, by comparing coordinates we have the following set of $d$ equations

$$e^{\mu^1}(t)=c_1(t) \mu_1^1+...+c_n(t) \mu_n^1+x_{11}$$
$$e^{\mu^2}(t)=c_1(t) \mu_1^2+...+c_n(t) \mu_n^2+x_{22}$$
$$\vdots$$
$$e^{\mu^d}(t)=c_1(t) \mu_1^d+...+c_n(t) \mu_n^d+x_{dd}.$$

This can be reduced to $n$ equations by utilizing the fact that $\sum\limits_{i=1}^{d} \mu_j^i x_{ii}=0$, as follows. Multiply each equation above by $\mu_1^i$ and sum (then repeat with $\mu_2^i,...,\mu_n^i)$

$$\sum\limits_{i=1}^d \mu_1^i e^{(\mu_1^i,...,\mu_n^i)}(t)=\sum\limits_{i=1}^d (\mu_1^i)^2 c_1(t) +\sum\limits_{i=1}^d \mu_1^i \mu_2^i c_2(t)+...+\sum\limits_{i=1}^d \mu_1^i \mu_n^i c_n(t)$$
$$\vdots$$
$$\sum\limits_{i=1}^d \mu_n^i e^{(\mu_1^i,...,\mu_n^i)}=\sum\limits_{i=1}^d \mu_1^i \mu_n^i c_1(t) +\sum\limits_{i=1}^d \mu_2^i \mu_n^i c_2(t)+...+\sum\limits_{i=1}^d (\mu_n^i)^2 c_n(t)$$

More cleanly this can be written as 
$$\begin{pmatrix}
\sum\limits_{\Lambda_\lambda}\mu_1\cdot e^{\mu}(t)\\
\vdots\\
\sum\limits_{\Lambda_\lambda}\mu_n\cdot e^{\mu}(t) \end{pmatrix}=
S(G,\lambda)
 \begin{pmatrix}
 c_1(t)\\c_2(t)\\ \vdots \\ c_n(t)
 \end{pmatrix}$$
 
 where
 $$S(G,\lambda):=\begin{pmatrix}
\sum\limits_{\Lambda_\lambda} \mu_1\cdot \mu_1 & \sum\limits_{\Lambda_\lambda} \mu_1\cdot \mu_2 &... & \sum\limits_{\Lambda_\lambda} \mu_1\cdot \mu_n\\
\sum\limits_{\Lambda_\lambda} \mu_1\cdot \mu_2 & \sum\limits_{\Lambda_\lambda} \mu_2\cdot \mu_2 &... & \sum\limits_{\Lambda_\lambda} \mu_2\cdot \mu_n\\
 \vdots &  \ddots & & \vdots\\
 \sum\limits_{\Lambda_\lambda} \mu_1\cdot \mu_n&... &  \sum\limits_{\Lambda_\lambda} \mu_{n-1}\cdot \mu_n & \sum\limits_{\Lambda_\lambda} \mu_n\cdot \mu_n \end{pmatrix}$$
 
 Then, we have that 
 $$\begin{pmatrix}
 c_1(t)\\c_2(t)\\ \vdots \\ c_n(t)
 \end{pmatrix}=S^{-1}(G,\lambda)
 \begin{pmatrix}
 \sum\limits_{\Lambda_\lambda} \mu_1 e^\mu (t)\\
\vdots\\
 \sum\limits_{\Lambda_\lambda}\mu_n e^\mu (t)
 \end{pmatrix}$$

  In the next section we calculate the matrix $S(G,\lambda)$ for the classical and exceptional simple algebraic groups. In the following sections, we continue the notation
 $$\Lambda_\lambda=\{(\mu_1,...\mu_n)|\ \mu_1\omega_1+...+\mu_n\omega_n\ is\ a\ weight\ of V_\lambda\}$$
 counted with multiplicity.
 
 \section{Main Result}
 Our main result will be calculating the matrix $S(G,\lambda)$ as defined in section 3, for the simple algebraic groups.  We use the convention that the Cartan matrix associated to the root system of $\fg$ is $A=(A_{ij})$, where $A_{ij}=\alpha_i(\check{\alpha_j})$. Then $A$ is a change-of-basis matrix for $\ft^*$ between the fundamental weights and the simple roots. Furthermore, $A$  satisfies the following properties 
 \begin{itemize}
     \item For diagonal entries $A_{ii}=2$
     \item For non-diagonal entries $A_{ij}\leq 0$
     \item $A_{ij}=0$ iff $A_{ji}=0$
     \item $A$ can be written as $DS$, where $D$ is a diagonal matrix, and $S$ is a symmetric matrix. 
 \end{itemize}
  Let $D$ be the diagonal matrix defined by $D_{ij}=\frac{\delta_{ij}}{2}(\alpha_i,\alpha_j)$, where if we realize the root system $R$ associated to $\fg$ as a set of vectors in a Euclidean space $E$, then $(\cdot,\cdot)$ is the standard inner product. In this framework we can write $A_{ij}=\alpha_i(\check{\alpha_j})=\frac{2(\alpha_i,\alpha_j)}{(\alpha_j,\alpha_j)}$ Then, writing $A=DS$, we find that the matrix $S$ has coordinate entries given by
  
  $$S_{ij}=\frac{4(\alpha_i,\alpha_j)}{(\alpha_i,\alpha_i)(\alpha_j,\alpha_j)}$$
  and is clearly symmetric.
  
  $(\cdot,\cdot)$ is an invariant bilinear form on $\ft^*$, normalized so that so that $(\alpha_i,\alpha_i)=2$ where $\alpha_i$ is the highest root. Note that under this formulation, if $G$ is of simply-laced type then $D$ is the identity matrix and $S$ is the Cartan matrix. We find that in general for a given simple group $G$ that $S(G,\lambda)$ is a multiple of $S$. Before stating our result precisely we fix the following notation. If $\alpha_j$ is any long simple root (for the simply laced case $\alpha_j$ can be any simple root), consider the corresponding fundamental weight $\omega_j$. Let $x_j(\lambda):=\sum\limits_{\mu\in\Lambda_\lambda}\mu_j^2$, where $\mu_j$ is the $j^{th}$ coordinate of the weight $\mu\in\Lambda_\lambda$ in the fundamental weight basis.
  
  \begin{proposition}
   Let $G$ be a simple algebraic group. Let $S(G,\lambda)$ be defined as in section 3. Set $x_j(\lambda):=\sum\limits_{\mu\in\Lambda_\lambda}\mu_j^2$ for a long root $\alpha_j$. This is independent of the choice of long root $\alpha_j$. Let $S$ be a symmetrization of the Cartan matrix as above. Then $S(G,\lambda)$ is a multiple of $S$. More precisely, 
   
       $$S(G,\lambda)=\frac{1}{2}x_j(\lambda)\cdot S$$
      
  \end{proposition}

\begin{proof}
 The proof will rely on the fact that the set of weights $\Lambda_\lambda$ of $V_\lambda$ is invariant under the action of the Weyl Group on $\ft^*$, i.e. for $w\in W$, $w.\Lambda_\lambda =\Lambda_\lambda$. The following Lemma is true for all simple groups. The following two lemmas are sufficient to prove the simply-laced case but also hold for the non-simply laced cases.
 
 \begin{lemma}
 For a given simple group G, if the Cartan matrix entry $A_{ij}=0$, i.e the nodes representing the simple roots $\alpha_i$ and $\alpha_j$ are not connected on the associated Dynkin diagram, then 
 $$\sum\limits_{\mu\in\Lambda_\lambda} \mu_i\cdot \mu_j=0,$$
 where $\mu=(\mu_1,...,\mu_n)$.
 
 \end{lemma}
 
\begin{proof}
 Consider the simple reflection $s_i$ acting on a weight $\mu=(\mu_1,...\mu_n)\in\Lambda_\lambda$. Then
 $$s_i(\mu)=(\mu_1,...\mu_n)-((\mu_1,...\mu_n)(\check{\alpha_i}))(\alpha_i)$$
 Where $(\mu_1,...\mu_n)(\check{\alpha_i})=(\mu_1\omega_1+...\mu_n\omega_n)(\check{\alpha_i})=\mu_i$. Using the Cartan matrix to write the simple roots $\alpha_i$ in the fundamental weight basis gives $\alpha_i=(A_{i,1},...,A_{i,n})$. Then the above reflection yields
 
 $$s_i(\mu)=(\mu_1,...\mu_n)-\mu_i(A_{i,1},...,A_{i,n})=(\mu_1-\mu_iA_{i1},...,\mu_n-\mu_iA_{in})$$
 
 Now note that $A_{ii}=2$ and $A_{ij}=0$. So the $i^{th}$ coordinate of $s_i(\mu)$ is $[s_i(\mu)]_i=\mu_i-\mu_iA_{ii}=-\mu_i$ and the $j^{th}$ coordinate of $s_i(\mu)$ is $[s_i(\mu)]_j=\mu_j-\mu_iA_{ij}=\mu_j$. Thus we find that
 
 $$\sum\limits_{\mu\in\Lambda_\lambda} \mu_i \mu_j=\sum\limits_{s_i(\mu)\in\Lambda_\lambda} \mu_i \mu_j=\sum\limits_{\mu\in\Lambda_\lambda} [s_i(\mu)]_i\cdot[s_i(\mu)]_j =\sum\limits_{\mu\in\Lambda_\lambda} -\mu_i \mu_j,$$
 
 by invariance of $\Lambda_\lambda$ under $s_i$. Thus,  the result follows.
 \end{proof}
 
 \begin{lemma}
 If simple roots $\alpha_i$ and $\alpha_j$ of $G$ are connected via the Dynkin diagram and have the same length then 
 $$\sum\limits_{\mu\in\Lambda_\lambda} (\mu_i)^2=\sum\limits_{\mu\in\Lambda_\lambda} (\mu_j)^2.$$ Furthermore, 
  $$\sum\limits_{\mu\in\Lambda_\lambda} \mu_i\cdot \mu_j=-\frac{1}{2}\sum\limits_{\mu\in\Lambda_\lambda} \mu_i\cdot \mu_i$$
 
 \end{lemma}
 
\begin{proof}
 
 Let $\alpha_i$ and $\alpha_j$ be roots of the same length whose corresponding nodes on the Dynkin diagram are connected. So $A_{ij}=A_{ji}=-1$. Then as above with $\mu=(\mu_1,...\mu_n)\in\Lambda_\lambda$, we have that $s_i(\mu)=(\mu_1-\mu_iA_{i1},...,\mu_n-\mu_iA_{in})$. Now consider 
 $$s_js_i(\mu)=((\mu_1-\mu_iA_{i1})-(\mu_j-\mu_iA_{ij})A_{j1},...,(\mu_n-\mu_iA_{in})-(\mu_j-\mu_iA_{ij})A_{jn})$$
 Thus, $[s_js_i(\mu)]_i=(\mu_i-\mu_iA_{ii})-(\mu_j-\mu_iA_{ij})A_{ji}=-\mu_i-(\mu_j+\mu_i)(-1)=\mu_j$. Thus,
 $$\sum\limits_{\Lambda_\lambda} \mu_i\cdot \mu_i=\sum\limits_{\Lambda_\lambda} [s_js_i(\mu)]_i\cdot [s_js_i(\mu)]_i=\sum\limits_{\Lambda_\lambda} \mu_j\cdot \mu_j$$
 The second part of the lemma follows from the fact that $[s_i(\mu)]_j=\mu_j-\mu_iA_{ij}$ with $A_{ij}=-1$. It follows that
$$\sum\limits_{\Lambda_\lambda} \mu_j^2=\sum\limits_{\Lambda_\lambda}[s_i(\mu)]_j^2=\sum\limits_{\Lambda_\lambda} (\mu_j+\mu_i)^2$$
Thus, $\sum\limits_{\Lambda_\lambda}\mu_i\cdot \mu_i=-2\sum\limits_{\Lambda_\lambda}\mu_i\cdot \mu_j$
 \end{proof}

\vspace{10 mm}

With the above results we see that for groups of simply-laced type that
$$\begin{pmatrix}
c_1(t)\\ \vdots \\ c_n(t)
\end{pmatrix}
=\frac{2}{\sum\limits_{\mu\in\Lambda_\lambda} \mu_i^2} A^{-1}
\begin{pmatrix}
\sum\limits_{\mu\in\Lambda_\lambda}\mu_1 e^\mu(t)\\
\vdots\\
\sum\limits_{\mu\in\Lambda_\lambda}\mu_n e^\mu(t) 
\end{pmatrix}$$
The inverses of the Cartan matrices for the simply laced root systems are in the Appendix.

\vspace{10 mm}

\subsection{Non-simply laced groups}
Recall that the roots systems of simple groups of type $B_n,C_n,G_2,F_4$ contain long and short simple roots. Our convention will be the same as in Bourbaki [Bo]. That is, for $B_n$ that $\alpha_1,...,\alpha_{n-1}$ are the long roots and $\alpha_n$ is short, for $C_n$ that $\alpha_1,...\alpha_{n-1}$ are short and $\alpha_n$ is long, for $G_2$ that $\alpha_1$ is short and $\alpha_2$ is long, and for $F_4$ that the first and second are long and that the third and fourth are short. 

\subsubsection{G of type B,C or F}
\begin{proposition}

 Let $G$ be a rank n simple group of types $B_n$, $C_n$, or $F_4$. For any long root $\alpha_i$, set $x=\sum\limits_{\Lambda_\lambda} \mu_i^2$.  If $\alpha_j$ is a short root, then $\sum\limits_{\mu\in\Lambda_\lambda}\mu_j^2=2x$, where $x$ is defined in \S 4.  If either or both of $\alpha_i$ and $\alpha_j$ are short, then $\sum\limits_{\mu\in\Lambda_\lambda}\mu_i\mu_j=-x$
 
\end{proposition}
\begin{proof}

Note that if $\alpha_i$ and $\alpha_j$ are both long roots, connected via the Dynkin diagram, then $A_{ij}=A_{ji}=-1$ So the same argument as in Lemma 4.3 shows that
$$\sum\limits_{\Lambda_\lambda}\mu_i^2=\sum\limits_{\Lambda_\lambda}\mu_j^2,$$
and that $\sum\limits_{\Lambda_\lambda}\mu_i\mu_j=-\frac{1}{2}\sum\limits_{\Lambda_\lambda}\mu_i^2$. The same is true for the short roots as $A_{ij}=A_{ji}=-1$ for connected short roots. So we need to show that if $\alpha_i$ and $\alpha_j$ are short and long roots respectively and connected via the Dynkin diagram, then $\sum\limits_{\Lambda_\lambda}\mu_i^2= 2x$, and that $\sum\limits_{\Lambda_\lambda}\mu_i\mu_j=-x$. To show this we first note that $A_{ij}=-1$ and $A_{ji}=-2$ and then compare $[s_i(\mu)]_i,[s_j(\mu)]_j,[s_j(\mu)]_i$ and $[s_i(\mu)]_j$. Note that $[s_i(\mu)]_i=-\mu_i$ and $s_j(\mu_j)=-\mu_j$ as before. Also, $[s_i(\mu)]_j=\mu_j-\mu_iA_{i,j}=\mu_j+\mu_i$ and $[s_j(\mu)]_i=\mu_i-\mu_jA_{ji}=\mu_i+2\mu_j$. Thus, we have that 

$$\sum\limits_{\Lambda_\lambda}\mu_i\mu_j=\sum\limits_{\Lambda_\lambda}[s_j(\mu)]_i\cdot [s_j(\mu)]_j=\sum\limits_{\Lambda_\lambda}(\mu_i+2\mu_j)(-\mu_j) =\sum\limits_{\Lambda_\lambda}-\mu_i\mu_j-2\mu_j^2$$

Thus $\sum\limits_{\Lambda_\lambda}\mu_i\mu_j=-\sum\limits_{\Lambda_\lambda}\mu_j^2=-x$. Applying, $s_i$ to $\mu$ gives

$$\sum\limits_{\Lambda_\lambda}\mu_i\mu_j=\sum\limits_{\Lambda_\lambda}[s_i(\mu)]_i\cdot[s_i(\mu)]_j=\sum\limits_{\Lambda_\lambda}-\mu_i\mu_j-\mu_i^2$$ 
Thus, $\sum\limits_{\Lambda_\lambda}\mu_i^2=2x$
\end{proof}

\vspace{10 mm}

So it follows that with $x=\sum\limits_{\Lambda_\lambda}\mu_j^2$, where $\alpha_j$ is a long root, then

$$S(B_n,\lambda)=\frac{x}{2}
\begin{pmatrix}
2&-1\\-1&2&-1\\&-1&\ddots\\&&&2&-1&\\&&&-1&2&-2\\&&&&-2&4
\end{pmatrix},
S(C_n,\lambda)=\frac{x}{2}
\begin{pmatrix}
4&-2\\-2&4&-2\\&-2&\ddots\\&&&4&-2&\\&&&-2&4&-2\\&&&&-2&2
\end{pmatrix}$$
$$S(F_4,\lambda)=\frac{x}{2}
\begin{pmatrix}
2&-1&0&0\\-1&2&-2&0\\0&-2&4&-2\\0&0&-2&4
\end{pmatrix}$$
We give inverses of these matrices in the appendix.

\subsubsection{G of type $G_2$}

\vspace{10 mm}

Let $\alpha_1$ be the short root, and $\alpha_2$ the long root of $G_2$.
\begin{proposition}
 $\sum\limits_{\Lambda_\lambda}\mu_1^2=-2\sum\limits_{\Lambda_\lambda}\mu_1\mu_2
=3\sum\limits_{\Lambda_\lambda}\mu_2^2$
\end{proposition}
\begin{proof}

Let $\mu=(\mu_1,\mu_2)\in\Lambda_\lambda$. Then since 
$A=\begin{pmatrix}
2&-1\\-3&2
\end{pmatrix}$, we find that $s_1(\mu)=(-\mu_1,\mu_1+\mu_2)$ and that $s_2(\mu)=(\mu_1+3\mu_2,-\mu_2)$. So,
$$\sum\limits_{\Lambda_\lambda}\mu_1^2=\sum\limits_{\Lambda_\lambda}(\mu_1+3\mu_2)^2$$
from which it follows that $\sum\limits_{\Lambda_\lambda}\mu_1\mu_2=-\frac{3}{2}\sum\limits_{\Lambda_\lambda}\mu_2^2$. Additionally, we have that 
$$\sum\limits_{\Lambda_\lambda}\mu_2^2=\sum\limits_{\Lambda_\lambda}(\mu_1+\mu_2)^2$$
from which we can see that $\sum\limits_{\Lambda_\lambda}\mu_1^2=-2\sum\limits_{\Lambda_\lambda}\mu_1\mu_2
=3\sum\limits_{\Lambda_\lambda}\mu_2^2$. Thus,
$$S(G_2,\lambda)=\frac{1}{2}\sum\limits_{\Lambda_\lambda}\mu_2^2
\begin{pmatrix}
6&-3\\-3&2
\end{pmatrix}$$
\end{proof}
In particular, we can solve for $c_1(t)$ and $c_2(t)$ as 
$$\begin{pmatrix}
c_1(t)\\c_2(t)
\end{pmatrix}=
(S(G_2,\lambda)^{-1}\begin{pmatrix}
\sum\limits_{\Lambda_\lambda}\mu_1e^\mu(t)\\\sum\limits_{\Lambda_\lambda}\mu_2e^\mu(t)
\end{pmatrix}$$
then, letting $x=\sum\limits_{\Lambda_\lambda}\mu_2^2$ we have that $S^{-1}(G,\lambda)=
\frac{2}{3x}\begin{pmatrix}
2&3\\3&6
\end{pmatrix}$. Thus,
$$c_1(t,\lambda)=\frac{2}{3x}\sum\limits_{\Lambda_\lambda}(2\mu_1+3\mu_2)e^\mu(t)$$
$$c_2(t,\lambda)=\frac{2}{3x}\sum\limits_{\Lambda_\lambda}(3\mu_1+6\mu_2)e^\mu(t)$$. 
\end{proof}
\section{Example($G=C_n$,Defining Reresentation)}
Consider $G=Sp(2n,\mathbb{C})$=\{$A\in GL(2n)|M=A^tMA$\} where $M=\begin{pmatrix}
0&I_n\\-I_n&0
\end{pmatrix}$ where $I_n$ is the $n\times n$ identity matrix, and $\mathfrak{sp}(2n,\mathbb{C})$=\{$X\in \mathfrak{gl}(2n)|X^tM+MX=0$\}. 

Let $\lambda=\omega_1$,the defining representation. Then we have that $\Lambda_\lambda$=\{$\pm\omega_1$ and $\pm(\omega_i-\omega_{i+1})$ for $1\leq i\leq n-1$\}. So, $x=\sum\limits_{\Lambda_\lambda} \mu_n^2=2$. Let $T=diag\{t_1,...,t_n,t_1^{-1},...,t_n^{-1}\}$. The simple roots are $\alpha_i=\epsilon_i-\epsilon_{i+1}$ for $1\leq i\leq n-1$ and $\alpha_n=2\epsilon_n$. The simple coroots in $\ft$ are then $\check{\alpha_i}=E_{i}-E_{i+1}-E_{n+i}+E_{n+i+1}$ for $1\leq 1\leq n-1$ and $\check{\alpha}_n$=$E_n-E_{2n}$ where $E_i$ is the diagonal matrix with a 1 in the $i^{th}$ slot and 0's elsewhere [FH]. In the orthogonal basis for $\ft$, $\omega_i=\epsilon_1+...+\epsilon_i$. Thus, the character $e^\mu(t)$ is given by $e^\mu(t)=t_1^{\mu_1+...\mu_n}\cdot t_2^{\mu_2+...+\mu_n}\cdot ...\cdot t_n^{\mu_n}$. Then, we have that
$$\begin{pmatrix}
c_1(t)\\ \vdots \\ c_n(t)
\end{pmatrix}=
\frac{1}{2}\begin{pmatrix}
1 & 1 & 1 &...& 1\\
1 & 2 & 2 &...&2\\
1& 2 & 3 &...&3\\
...&...&...&...&...\\
1 & 2 & 3 &... & n
\end{pmatrix}
\begin{pmatrix}
t_1-t_1^{-1}-t_2+t_2^{-1}\\
t_2-t_2^{-1}-t_3+t_3^{-1}\\
\vdots\\ t_{n-1}-t_{n-1}^{-1}-t_n+t_n^{-1}\\
t_n-t_n^{-1}
\end{pmatrix}$$
which gives 
$$\begin{pmatrix}
c_1(t)\\ \vdots \\ c_n(t)
\end{pmatrix}=\frac{1}{2}\begin{pmatrix}
t_1-t_1^{-1}\\ \vdots \\ t_{n-1}-t_{n-1}^{-1} \\ t_1-t_1^{-1}+...+t_n-t_n^-{1}
\end{pmatrix}$$
Thus, $$\theta_\lambda(t)=c_1(t)\check{\alpha}_1+...+c_n(t)\check{\alpha}_n=diag(\frac{t_1-t_1^{-1}}{2},...,\frac{t_n-t_n^{-1}}{2},-\frac{t_1-t_1^{-1}}{2},...,-\frac{t_n-t_n^{-1}}{2}).$$
Note that this is equivalent to the Cayley transform as in \S 6 of [Kum]. Similiar results hold for $\theta_{\omega_{1}}(t)$ for the standard maximal tori of $SO(2n,\mathbb{C})$ and $SO(2n,\mathbb{C})$.

\appendix
\section{inverse of the cartan matrices and their symmetrizations S}
 The the inverses of the Cartan matrices for $A_n,D_n,E_6,E_7,E_8$ respectively have the form (as in [Rosenfeld]))
$$\frac{1}{n+1}
\begin{pmatrix}
n&n-1&n-2&...&3&2&1\\
n-1&2(n-1)&2(n-3)&...&6&4&2\\
n-2&2(n-2)&3(n-2)&...&9&6&3\\
...&...&...&...&...&...&...\\
2&4&6&...&(2n-2)&2(n-1)&n-1\\
1&2&3&...&n-2&n-1&n
\end{pmatrix}
,$$
$$\begin{pmatrix}
1&1&1&...&1&\frac{1}{2}&\frac{1}{2}\\
1&2&2&...&2&1&1\\
1&2&3&...&3&\frac{3}{2}&\frac{3}{2}\\
...&...&...&...&...&...&...\\
1&2&3&...&n-2&\frac{n-2}{2}&\frac{n-2}{2}\\
\frac{1}{2}&1&\frac{3}{2}&...&\frac{n-2}{2}&\frac{n}{4}&\frac{n-2}{4}\\
\frac{1}{2}&1&\frac{3}{2}&...&\frac{n-2}{2}&\frac{n-2}{4}&\frac{n}{4}
\end{pmatrix}$$

$$\begin{pmatrix}
\frac{4}{3}&1&\frac{5}{3}&2&\frac{4}{3}&\frac{2}{3}\\
1&2&2&3&2&1\\
\frac{5}{3}&2&\frac{10}{3}&4&\frac{8}{3}&\frac{4}{3}\\
2&3&4&6&4&2\\
\frac{4}{3}&2&\frac{8}{3}&4&\frac{10}{3}&\frac{5}{3}\\
\frac{2}{3}&1&\frac{4}{3}&2&\frac{5}{3}&\frac{4}{3}\\
\end{pmatrix},
\begin{pmatrix}
2&2&3&4&3&2&1\\
2&\frac{2}{2}&4&6&\frac{9}{2}&3&\frac{3}{2}\\
3&4&6&8&6&4&2\\
4&6&8&12&9&6&3\\
3&\frac{9}{2}&6&9&\frac{15}{2}&5&\frac{5}{2}\\
2&3&4&6&5&4&2\\
1&\frac{3}{2}&2&3&\frac{5}{2}&2&\frac{3}{2}
\end{pmatrix},
\begin{pmatrix}
4&5&7&10&8&6&4&2\\
5&8&10&15&12&9&6&3\\
7&10&14&20&16&12&8&4\\
10&15&20&30&24&18&12&6\\
8&12&16&24&20&15&10&5\\
6&9&12&18&15&12&8&4\\
4&6&8&12&10&8&6&3\\
2&3&4&6&5&4&3&2
\end{pmatrix}$$
The inverse of the matrix $S$ for types $C_n,B_n,G_2,F_4$ have the form
$$\frac{1}{2}\begin{pmatrix}
1 & 1 & 1 &...& 1\\
1 & 2 & 2 &...&2\\
1& 2 & 3 &...&3\\
...&...&...&...&...\\
1 & 2 & 3 &... & n
\end{pmatrix},\frac{1}{2}
\begin{pmatrix}
2&2&2&...&2&1\\
2&4&4&...&4&2\\
2&4&6&...&6&3\\
...&...&...&...&...&...\\
2&4&6&...&2(n-1)&n-1\\
1&2&3&...&n-1&2
\end{pmatrix},
\begin{pmatrix}
\frac{2}{3}&1\\1&2
\end{pmatrix},
\begin{pmatrix}
2&3&2&1\\3&6&4&2\\2&4&3&\frac{3}{2}\\1&2&\frac{3}{2}&1
\end{pmatrix}$$
\bibliographystyle{plain}
\def\noopsort#1{}

\end{document}